%% file: p2p2-II9.2.tex
 \documentclass[11pt]{amsart}
 \usepackage{amsmath}
 \usepackage{amssymb}
 \usepackage{amsfonts}
 \usepackage{amsthm}
 \usepackage{verbatim}
 \usepackage{amscd}
 \usepackage{cite}
 \usepackage{leftidx}
 \usepackage{enumerate}
 \usepackage{txfonts}
 \usepackage{manfnt}
 \usepackage{amscd}
 \usepackage[mathscr]{eucal}
 \usepackage{hyperref}
 \usepackage{datetime2}
 \usepackage{graphics}
 \usepackage{mathrsfs}  
 \usepackage{tikz-cd}
 \usepackage{enumitem}
 \setlist{nolistsep}
 \usepackage{mathpazo}
 \def\Num{\mathrm{Num}}
 \def\NS{\mathrm{NS}}

 \def\bZ{\mathbb Z}
 \def\bP{\mathbb P}
 \def\ar[r]{\to}
 \def\cI{\mathcal I}
 \def\cA{\mathcal A}
 \def\cO{\mathcal O}

 \DeclareMathOperator{\Pic}{Pic}

 \newtheorem{thm}{Theorem} 

 \newtheorem*{thm*}{Theorem}
 \newtheorem*{prop*}{Proposition}
 \newtheorem{cor}[thm]{Corollary}
 \newtheorem*{cor*}{Corollary}

 \newtheorem{lem}[thm]{Lemma}
 \newtheorem*{lem*}{Lemma}

 \newtheorem*{claim*}{Claim}
 \newtheorem{prop}[thm]{Proposition}
 
 \theoremstyle{remark}
 
 \newtheorem{rem}[thm]{Remark}
 \newtheorem*{rem*}{Remark}
 \newtheorem{crit-rem}[thm]{Critical remark}

 \newtheorem{example}[thm]{Example}
 \newtheorem*{example*}{Example}
   
 \newtheorem*{defn*}{Definition}
 \newtheorem*{con*}{Conjecture}
 \input{./myheader3.tex}

\begin{document} 
 	\title{A  rank-$2$ vector bundle on ${\P}^2\times {\P}^2$
 	and projective geometry of nonclassical Enriques surfaces in characteristic 2 }
 	\author 
 	{Ziv Ran and J\"urgen Rathmann}
 %
 %
 	\thanks{arxiv.org 2504.16174 }
 	\date {\DTMnow}
 %
 %
 	\address {\nl ZR: UCR Math Dept. \ 
 	ziv.ran @  ucr.edu\nl
 	JR:jk.rathmann@gmail.com
 	}
 	
 	 \subjclass[2010]{14n05, 14j60}
 	\begin{abstract}
 We study zero-sets pf  a particular family rank-$2$ indecomposable vector bundles
 $E$  on $\P^2\times\P^2$ 
in  characteristic $2$, introduced in \cite{p2p2-I}.
 We show that the zero-sets of a suitable twist of $E$ form
 a family of nonclassical smooth Enriques surfaces of bidegree (4, 4) whose  general member is
 ordinary in the sense that Frobenius acts isomorphically on $H^1$, and which admits a divisor
 consisting of smooth supersingular surfaces (Frobenius acts as zero).
We show further that every nonclassical Enriques surface
 of bidegree (4, 4) that is bilinearly normal arises as a zero-set in this way.

 		\end{abstract}
 			\maketitle
 			A well-known result of Serre asserts that a codimension $2$ locally complete intersection in projective space
 			is the zero scheme of a section of a rank-$2$ vector bundle if and only if its canonical line bundle is
 			the restriction of a line bundle on the ambient space. Serre's result is also
 			valid for certain other ambient spaces.
 			When the canonical bundle is trivial, the condition is automatically satisfied.
 			Among algebraic surfaces, there are three well-known classes of minimal surfaces with a trivial canonical bundle, 
 			and two of them 
 			have models in $\bP^4$ with corresponding bundles:
 			\begin{enumerate}
 			\item
 			Abelian surfaces can be embedded into $\bP^4$ by an ample $(1,5)$-polarization, and the
 			corresponding bundle has first been described by Horrocks and Mumford in 1973.
 			\item
 			K3 surfaces with an ample divisor of degree $6$ can be embedded in $\bP^4$ as  a complete
 			intersection of a quadric and a cubic hypersurface, hence the corresponding vector bundle splits.
 			\item
 			In characteristic 2 only,  there exist non-classical (ordinary and supersingular)
 			 Enriques surfaces with a trivial canonical bundle;
 			however, the self-intersection formula \cite[App. A]{hart} shows that they cannot be embedded into $\bP^4$.
 			\end{enumerate}
 			In light of this one can naturally ask whether non-classical Enriques surfaces
 			can be embedded in other 'standard' homogeneous rational 4-folds
 			and give rise to new rank-2 vector bundles..

In \cite{p2p2-I} we studied
 		families of rank-$n$ bundles $E$ on $\P^n\times\P^n$
 		in all positive characteristics, with special attention to the case $n=2$ in
 		characteristic 2. In this paper we continue the study of the latter case,
focusing on the associated zero-sets of $E$ and some of its twists. 
The main results may be summarized as follows.
 		\begin{enumerate}
 		\item (See   \S\ref{geometry}) The zero set $Y$ of a general section of a suitable  twist of $E$ is
 		a  smooth nonclassical Enriques surface; thus $Y$ has trivial canonical bundle and irregularity 1. 
 		(See \S\ref{ss-sec}) The general smooth zero set is 
 		ordinary (\emph{singular} in the sense of Bombieri-Mumford,i.e.  
 		Frobenius acts isomorphically on $H^1(\O_Y)$) and there is a nontrivial
 		codimension-1 subfamily of \emph{supersingular} smooth zero sets
 		 (Frobenius is zero on $H^1(\O_Y)$). (See \S\ref{moduli-sec}) There is a close relationship
 		 between the moduli of the bundle $E$ and that of the surface $Y$.
 		\item (See \S\ref{reducible-sec}) The same twist has special sections with zero-set that is 
 		a normal-crossing surface  of the form $Y_1\cup_WY_2$ where $Y_1$
 		is isomorphic to $ \P^2$ , 
 		$Y_2$ is an elliptic ruled surface and $W$ is a smooth  supersingular  elliptic curve
 		\item (See \S\ref{embeddings}) Any nonsingular nonclassical ($K=0$) Enriques surface in $\P^2\times\P^2$ that is 'bilinearly
 		normal',  i.e. does not lift to a higher dimensional $\P^r\times\P^s$, occurs as one of the
 	zero-sets $Y$ above.
 		
 		\end{enumerate}
 		An interesting question the we are unable to answer in whether,
 		at least in the ordinary case,  the canonical K3 
 		double cover
 		of our Enriques surfaces embed  compatibly in $\P^2\times\P^2\times\P^1$ and if so
 		whether they are Pfaffian, corresponding to a vector bundle that is somehow related to our bundle.
 		\par  Eisenbud, Popescu and Walter \cite{ eisenbud-p-w-enriques}, \S 8 have constructed
 		nonclassical Enriques surfaces of degree 10
 		 in $\P^5$ as (higher) degeneracy loci (specifically, intersection loci of a pair of Lagrangian
 		 subbundles of $\wedge^3H^0(\O_{\P^5}(1)$).
 		 
 	
 \bigskip
 		
 	\section*{Preliminaries}\subsection{Projective planes} \label{projective-planes}
 	Consider a copy of $\P^2$ with hyperplane class $L$ and let $V_L=H^0(\O(L))^*$ so we have an exact sequence
 	\[\exseq{\O(-L)}{V_L}{Q_L}\]
 	where $Q_L=T_{\P^2}(-L)\cong \Omega^1_{\P^2}(2L)$ is the universal quotient bundle.
 	Any nonzero section of $Q_L$ vanishes at a unique point $p\in\P^2$ and yields an exact sequence
 	\[\exseq{\O}{Q_L}{\I_p(L)}.\]
 	Note that $\wedge^2Q_L=\O(L)$ whence 
 	an isomorphism
 	\[Q_L\simeq \Hom(Q_L, \O(L)),\]
 	\[v\mapsto v^t\]
 	and a skew-symmetric  pairing
 	\[Q_L\times Q_L\to \O(L),\]
 	\[(v, w)\mapsto \carets{v, w}.\]
 	Given $\phi\in H^0(Q_L)$ vanishing at $p$ and $\psi\in H^0(Q_L)=H^0(\Hom(Q_L, \O(L)))$,
 	if $\psi$ vanishes at $q\neq p$, then the composite $\carets{v, w}=\psi\circ\phi\in H^0(\O(L))$ 
 	vanishes on the line spanned by $p,q$, while if $q=p$, i.e. $\psi=\phi^t$, 
 	then the composite $\psi\circ\phi=0$.
 	For any nonzero section $\phi$ vanishing at $p$, the map $\psi\mapsto\psi\circ\phi$ yields a surjection
 	$H^0(Q_L)\to H^0(\Hom(Q_L, \I_p(L)))= H^0(\I_p(L))$ with kernel generated by $\phi^t$. 
 	\par Note  under the above isomorphism the composition map
 	\[H^0(Q_L)\times H^0(\Hom(Q_L, \O(L))\to H^0(\O(L))\]
 	is compatible up to scalar multiples with the wedge product
 	\[\wedge^2V\to V^*\]
 	and in particular we have $\phi^t\circ\phi=0$.
 	For $v\in V$ we denote the 2-dimensional subspace 
 	$v\wedge V\subset V^*$ by $U_v$.
 	\par
 	\subsection{Products and maps}
 	Now consider two copies of $\P^2$, denoted $\P^2_L, \P^2_h$ with respective line generators $L, h$ and
 	respective tangent bundles $T_L, T_h$, and quotient bundles $Q_L=T_L(-L), Q_h=T_h(-h)$ as above. 
 	Setting $V_L=H^0(\O(L))^*, V_h=H^0(\O(h))^*$, we then have
 	\[H^0(Q_L\otimes Q_h)=V_L\otimes V_h\]
 	and this 9-dimensional vector space is endowed with an increasing filtration (by cones)
 	\[R_1\subset R_2\subset R_3=V_L\otimes V_h\]
 	where $R_i$ denotes the elements expressible as a sum of $i$ decomposable elements.
 	The component $R_i$
 	can be identified as the locus of $3\times 3$ matrices of rank $i$ or less.
 	Note that $R_1$ has dimension $3+3-1=5$ while $R_2$ has dimension 8.
 	A section in $R_2\setminus  R_1$ vanishes at a unique point $(p, p')\in\P^2_L\times \P^2_h$,
 	while any section in $R_3\setminus R_2$ is nowhere vanishing.\par
 	Now consider a rank-1 element $\phi=v\otimes w\in V_L\otimes V_h$. 
 	It is easy to check that the map $\phi^t :V_L\otimes V_h\to V_L^*\otimes V_h^*=\O(L+h)$ 
 	has kernel $v\otimes V_h+V_L\otimes w$ (a 5-dimensional subspace) and image
 	the 4-dimensional subspace $U_v\otimes U_w$, which coincides 
 	with the set of bilinear forms vanishing on $Z_{v, w}:=<v>\times\P^2_h\cup\P^2_L\times<w>$.\par
 	Next, for a rank-2 element $\phi_1+\phi_2=v_1\otimes w_1+v_2\otimes w_2$, we have that
 	\[(v_1\otimes V_h+V_L\otimes w_1)\cap (v_2\otimes V_h+V_L\otimes w_2)=k v_1\otimes w_2+k v_2\otimes w_1\]
 	is 2 dimensional while 
 	\[U_{v_1}\otimes U_{w_1}\cap U_{v_2}\otimes U_{w_2}=k \carets{v_1,v_2}\otimes \carets{w_1, w_2}, 
k= \mathrm{ground field} 	\]
 	is 1-dimensional, hence  $U_{v_1}\otimes U_{w_1}+ U_{v_2}\otimes U_{w_2}$ is 7-dimensional.
 	Thus the image of $\phi_1^t+\phi_2^t$ is  7-dimensional, and consists exactly
 	of the forms vanishing on the 2 points
 	$Z_{v_1, w_2}\cap Z_{v_2, w_1}=(<v_1>, <w_2>)\cup(<v_2>, <w_1>)$.\par
 	Similarly, for a rank-3 element $\phi= \phi_1+\phi_2 +\phi_3$, the image of  $\phi^t$
 	is at least 8-dimensional, and since we know $\phi^t\circ\phi=0$, it follows that
 	$\ker(\phi^t)=<\phi>$ and moreover the  image of $\phi^t$ is base-point free.
 	
 		 \subsection{Principal parts and smooth zero-sets}\label{principal-parts}
 		 Compare (EGA IV 16.3.2,16.7.2.1).
 		 For a bundle $E$ on a smooth variety $X$, we denote by $P(E)$ its bundle
 		 of first-order principal parts,  defined as $p_{2*}p_1^*E$ where $p_1, p_2$
 		  		  are the (bijective) projections $\Spec(\O_{X\times X}/\I_\Delta^2)\to X, 
 		  		  \Delta=$ diagonal.
 		 Recall  that there is a canonical short exact sequence
 		  \begin{equation*}
 		  0\to E\otimes\Omega^1_{\bP^2\times\bP^2}\to P(E)\to E\to 0.
 		  \end{equation*}
 		 
%
 		 This sequence admits a canonical additive splitting (and in particular induces a surjection
 		 $H^0(P(E))\to H^0(E)$), where a section $\O\to E$ maps to the following section of
 		 $P(E)$, called its canonical lift:
 		 \[\O_X\to P(\O_X)=\O_X\oplus \Omega_X\to P(E)\]
 		 where the left map is $f\mapsto (f, df)$.\par
 		 In positive characteristic, a globally generated bundle need not have a smooth general zero-set
 		 (cf. \cite{hartshorne-rank2-p3}, Prop. 1.4).
 		 However the following criterion in terms of principal parts holds:
 		 \begin{lem}\label{principal-lem}
 		 Let $E$ be a vector bundle over a smooth variety $X$
 		 and $V\subset H^0(E)$  a finite-dimensional subspace whose canonical lift
 		 generates $P(E)$. 
 		 Then the zero-set of a general element of $V$
 		 is smooth.
 		 \end{lem}
 		  \begin{proof}
 		 Let us say that a section of $P(E)$ is degenerate at a point $p\in X$
 		 if the image in E vanishes at p and the induced element of $\Omega\otimes E$ at $p$
 		 is a degenerate tensor (corresponds to a matrix of non-maximal rank).
 		 If $P(E)$ is generated by $V$, the set of sections in $V$ degenerate at a given point $p\in X$ has codimension
 		 at least $1+dim(X)$, hence the set of sections degenerate somewhere is a proper subset. 
 		 Therefore the zero set of the image in $E$ of a nowhere degenerate section is smooth.  
 		 \end{proof}
 		 \begin{rem}
 		 The Lemma implies the usual Bertini Theorem that the general hyperplane section of a smooth 
 		 projective variety $X$ is smooth. This is because
 		 \[P(\O_{\P^n}(1))=H^0(\O_{\P^n}(1))\otimes \O.\]
 		This sheaf is globally generated, hence so is its quotient $P_X(\O_X(1))$.
 		 \end{rem}
 		 
 		  \subsection{Supersingular plane cubics}
 		 
 		 Recall that an elliptic curve $C$ over a field of positive characteristic is called \emph{supersingular}, if the
 		 action of Frobenius on $H^1(\cO_C)$ is zero.
 		 Supersingularity of elliptic curves can be detected from the equation in a
 		 plane embedding as a cubic curve. In the following we need an analogous result 
 		 for arbitrary subschemes defined by a cubic equation. The proof from
 		 \cite[I\!V~4.21]{hart} applies to this
 		 more general situation without modification and yields the next result.
 		 
 		 \begin{prop}\label{prop16}
 		 Let $X\subset\bP^2=\Proj (k[x,y,z])$ be the subvariety defined by a homogeneous cubic
		 equation $f(x,y,z)=0$ in characteristic $2$. Then the morphism 
 		 $F^*\colon H^1(X,\cO_X)\to H^1(X,\cO_X)$
 		 induced by Frobenius is $0$ if and only if the coefficient of $xyz$ in $f$ is $0$.
 		 \end{prop}
 		 
 		 Note that degree $3$ polynomials in characteristic $2$ with vanishing $xyz$-term are 
 		 the image of $\big(F^*H^0(\cO_{\bP^2}(1))\big)\otimes H^0(\cO_{\bP^2}(1))$, hence can be described
 		 without reference to a particular choice of coordinates.
 		 
 		 \begin{example}\label{ex17}
 		 We can group the subvarieties defined by cubic equations as follows:
 		 \medskip
 		 
 		 {\par\noindent
 		 \begin{tabular}{l | l }
 		 $F^*=0$ &$F^*=id_X$ \\
 		 \hline
 		 $\bullet$ supersingular elliptic curve &$\bullet$ ordinary elliptic curve \\
 		 $\bullet$ cuspidal rational cubic & $\bullet$ nodal rational cubic  \\
 		 $\bullet$ conic and a line tangent to it & $\bullet$ conic and a line intersecting transversely  \\
 		 $\bullet$ three lines intersecting in a point & $\bullet$ three lines with three points of intersection \\
 		 $\bullet$ a double line and another line  & \\
 		 $\bullet$ a triple line  & \\
 		 \end{tabular}
 		 }
 		 \end{example}
 		 \medskip

\bigskip

\vfil\eject

 
\subsection{Review of results from \cite{p2p2-I}}
Let $\P^2_L, \P^2_h$ denote copies of the (polarized) projective plane
 in characteristic $p>0$ (in the present paper, $p=2$), let $\cA\subset\P^2_L\times \P^2_h$
 be a smooth divisor of class $L+h$, and let $F$ denote the Frobenius map.
 The rank-2 bundle $E$ on $\P^2_L\times\P^2_h$ is defined by the exact sequence
 \eqspl{e-via-degeneracy}{
 \exseq{E}{p_h^*F^*Q_h}{\O_\cA(pL)}
 } 
 Alternatively $E$ can be defined as the cohomology of the following monad 
 \eqspl{e-via-monad}{
 \exseq{\O}{Q_L\otimes Q_h}{\O(L+h)}
 }
 (Recall that a \emph{monad} is a complex of locally free sheaves of the form
 \[0\to A\to B\to C\to 0\]
 which is exact at $A$ and $C$, and such that the map $A\to B$
 embeds $A$ as a subbundle.)  

The total Chern class of $E$ is give by
\eqspl{chern}{
c(E)=1+3(L+h)+4L^2+5Lh+4h^2.
}

 \section{geometry of zero sets}\label{geometry}
  
  The next four sections investigate the zero sets of sections of $E$. Key results are:
 \begin{enumerate}
 \item
 The general zero set is nonsingular (Corollary \ref{cor29})
 \item
 the quantitative invariants of smooth zero sets, specifying them as
 nonclassical Enriques surfaces (Theorem \ref{thm43})
 \item
 an analysis of certain reducible zero sets (§ \ref{reducible-sec}) 
 \item
 a criterion to distinguish ordinary and supersingular (smooth) zero sets (Theorem \ref{thm49})
  \end{enumerate}
 \medskip
 
 Section \ref{moduli-sec} discusses moduli of $E$ and relates them to moduli of its zero sets,
 non-classical Enriques surfaces.
 
 The final section \ref{embeddings} shows how to recover the monad from an embedded
 nonclassical Enriques surface of bidegree $(4,4)$.
 
  \subsection{Minimal zero sets}
  
  By the results of \cite{p2p2-I}, § 3, $E(-L)$ and $E(-h)$ are minimal twists with nonzero sections.
  They satisfy $h^0=3$, $h^1=0$ and every zero-scheme is two-dimensional. The analysis
  of their zero-schemes is the same in both cases, and we will discuss below only $E(-L)$.
  
  If $s\in H^0(E(-L))$ has zero scheme $Y$, then $Y$ has class
  \begin{equation*}
  [Y] = 2L^2+2Lh+4h^2,\quad \omega_Y=\cO_Y(-2L)
  \end{equation*}
  
  \begin{prop}\label{prop24}
  Let $s\in H^0(E(-L))$ be a section with zero scheme $Y$. Then $Y=Y_1\cup Y_2$ where:
  \begin{enumerate}
  \item[\textup{(1)}]
   $Y_1=2Z_1\subset\cA$ is a double structure on $Z_1=\pi_2^*(L_1)$ for some line $L\subset\bP^2_L$.
  \item[\textup{(2)}]
   $Y_2=\bP^2\times N_2(p)$ is a multiplicity-4 structure where $N_2(p)$ is the infinitesimal
  neighborhood of a point $p\in\bP^2$ of length 4 with ideal $(x^2,y^2)$ \emph{(}where $p$ is defined by 
  the vanishing of two local coordinates $x,y$\emph{)}.
  \end{enumerate}
  \end{prop}
  \begin{proof}
  First of all, the map $F^*Q\to E(-L)$ induces an isomorphism on sections, and moreover, we have
  $E(-L)\otimes\cO_\cA\cong \cO(2L)\oplus\cO(3h-L)$. Hence the sections of $F^*Q$ restrict on
  $\cA$ to the squares in $H^0(\cO_\cA(2L))$.
  This means that $Y$ must intersect $\cA$ in a surface of the form $2Z_1$ where $Z_1=pr_2^{-1}(L_1)$
  for some line $L\subset \bP^2_L$. Note that $Z_1$ is isomorphic to $\bP(\cO(1)\oplus\cO)$, a
  rational ruled surface of type $F_1$, and moreover
  \begin{equation*}
  Y=2Z_1+Y_2
  \end{equation*}
  where $\vert Y_2\vert = 4h^2=4[\bP^2_L\times\{x\}]$. In fact, in view of the exact sequence
  \begin{equation*}
  0\to F^*Q_h\to E(-L)\to\cO_\cA(3h-L),
  \end{equation*}
  $Y_2$ is also a zero-set of $F^*Q$, and we conclude that $Y_2$ is a multiplicity-$4$ structure
  of the form $\bP^2_L\times N_2(p)$ where $N_2(p)$ is an infinitesimal neighborhood of a point
  $p\in\bP^2_h$ with ideal of the form $(x^2,y^2)$.
  \end{proof}
 Next an easy consequence of the minimality of $E(-L), E(-h)$:
  \begin{cor}\label{multi-cor}
  The map of sections
  \begin{equation*}
  \psi\colon\Big(H^0(E(-L))\otimes H^0(\cO(L))\Big)
  \oplus
  \Big(H^0((E(-h)))\otimes H^0(\cO(h))\Big)
  \to H^0(E)
  \end{equation*} 
  is injective, and its image is $18$-dimensional.
  \end{cor}
  \begin{proof}
  In principle the statement can be read off from the computer calculations of the Corollary
  in the appendix to \cite{p2p2-I}, but we would like to give an independent proof.
  
  Since $H^0(E(-L-h))=0$, the restriction of sections
  $H^0(E) \to H^0(E\otimes\cO_\cA)$ is injective, and it suffices to
  validate the statement on $\cA$.
  
  By \cite{p2p2-I}, \S 2.5, sections of $E\otimes\cO_\cA$ correspond to
  sections of $\cO_{\bP^2}(3L)\oplus\cO_{\bP^2}(3h)$. Proposition \ref{prop24} 
  implies that the sections of $E(-L)$ correspond to the squares
  of linear polynomials in $H^0\cO(2L)$, hence the sections in
  $H^0(E(-L))\otimes H^0(\cO(L))$ map to cubic
  polynomials in $H^0( \cO(3L))$ without an $abc$-term.
  Similarly, sections of $H^0(E(-h))\otimes H^0(\cO(h))$
  map to cubic polynomials in $H^0 (\cO(3h))$ without
  an $xyz$-term.
  
  Injectivity of $\psi$ is now clear, and the dimension of the image can be read off
  from the known dimensions of the terms on the left.
  \end{proof}
  
  \begin{rem}
  We have inclusions
  \begin{equation*}
  \im(\psi)\subset H^0(E) \subset H^0(E\otimes\cO_\cA)\cong H^0(\cO(3L))\oplus H^0(\cO(3h))
  \end{equation*}
  of codimension $1$ vector spaces.
  Let $(s_L,s_h)\in H^0(\cO(3L))\oplus H^0(\cO(3h))$ be a section.
  
  The space on the left includes all sections $s_L, s_h$ where the corresponding polynomials
  have no terms $abc$, $xyz$.
  
  We will see later in \S \ref{ss-sec} that a 
  section $(s_L,s_h)\in H^0(E)$ but not in the image of $\psi$
  must have non-vanishing $abc$- and $xyz$-terms. 
  \end{rem}
 
\subsection{Sections of $E$}

 We start our investigation of the sections by considering a subset of them, namely those in the
 image of
 \begin{equation*}
 \psi\colon \Big(H^0(E(-L))\otimes H^0(\cO(L))\Big)
 \oplus
 \Big(H^0(E(-h))\otimes H^0(\cO(h))\Big)
 \to H^0E
 \end{equation*} 
 
 \begin{thm}\label{thm41}
 The general section in the image of $\psi$ has a smooth zero scheme.
 \end{thm}
 \begin{proof}
 This is a local question, and we will treat the cases $x\in\cA$ and $x\not\in\cA$ separately.
 
 \textbf{Case 1} ($x\in\cA$):
 On $\cA$, $E$ splits as $\cO_\cA(3L)\oplus\cO_\cA(3h)$, and the sections in the image
 of $\psi$ correspond to those in the direct sum of $F^*H^0(\cO(L))\otimes H^0(\cO(L))$
 and $F^*H^0(\cO(h))\otimes H^0(\cO(h))$.
 
 Regarding the first summand, we know that the linear system corresponding to
 $F^*H^0(\cO(L))$ is basepointfree, while $H^0(\cO(L))$ is very ample, hence by
 \cite[3.5]{Klei} the general section in $F^*H^0(\cO(L))\otimes H^0(\cO(L))$ is nonsingular on $\bP^2$.
 A similar argument holds for the second summand, and taken together they
 imply that the general section of $E$ is nonsingular on $\cA$.
 
 \textbf{Case 2} ($x\not\in\cA$):
 We claim that the stalks of $P(E)$ outside $\cA$ are generated by pull-back of sections of 
 $R_L$ ($:= pr_{1,*}E_0$) 
 and $R_h$ ($:= pr_{2,*}E_0$) via
 \begin{equation*}
 pr_1^*P(R_L)\oplus pr_2^*P(R_h) \to P(pr_1^*R_L)\oplus P(pr_2^*R_h) \to P(E).
 \end{equation*}
 If this holds, then we can apply Lemma \ref{principal-lem} to conclude
 that the general section of $E$ will be nonsingular outside $\cA$.
 
 Regarding the claim, recall \S \ref{principal-parts} that there is a canonical short exact sequence
 \begin{equation*}
 0\to E\otimes\Omega^1_{\bP^2\times\bP^2}\to P(E)\to E\to 0
 \end{equation*}
 where, for any $\O_X$-module $E$, $P(E)$ is defined as $p_{2*}p_1^*E$ where $p_1, p_2$
 are the (bijective) projections $\Spec(\O_{X\times X}/\I_\Delta^2)\to X$.
The sequence is split as a sequence of sheaves of abelian groups. The splitting is defined
 by sending a local section $s$ to $p_1\inv s$ where $p_1$  as above  is the 
projection.\par 

The cotangent bundle of $\bP^2\times\bP^2$ splits as
 $\Omega^1_{\bP^2\times\bP^2}\cong pr_1^*\Omega^1_{\bP^2}\oplus pr_2^*\Omega^1_{\bP^2}$
 (EGA IV 16.4.23), hence we conclude that
 \begin{equation*}
 P(E)\cong (E\otimes pr_1^*\Omega^1_{\bP^2})\oplus (E\otimes pr_2^*\Omega^1_{\bP^2})\oplus E
 \end{equation*}
 as sheaves of abelian groups.  
 
Since the injections $pr_1^*R_L\to E$ (resp. $pr_2^*R_h\to E$) are isomorphisms outside $\cA$
(the cokernel is supported on $\cA$), this decomposition for $P(E)$ holds outside $\cA$ 
 also, if $E$ is replaced by $pr_1^*R_L$ or by $pr_2^*R_h$ on the right side.

%

Finally, consider the bundle $R_L=F^*Q\otimes\cO_{\bP^2}(1)$. 
$F^*Q$ is globally generated on $\bP^2$, and
 $\cO_{\bP^2}(1)$ is very ample. Therefore the canonical lifts (see Lemma 3) of the sections of
 $R_L=F^*Q\otimes\cO(1)$ generate the stalks of the bundle
 $P(R_L)\cong (R_L\otimes\Omega^1_{\bP^2})\oplus R_L$ (split as sheaves of abelian groups)
 of principal parts of $R_L$ on $\bP^2$.
 
 Hence the pullback of sections under the map
 \begin{equation*}
 pr_1^*P(R_L)\to P(pr_1^*R_L)\to pr_1^*(R_L\otimes\Omega^1_{\bP^2})\oplus pr_1^*(R_L)
 \end{equation*}
 generates the stalks of the sheaf of the right outside $\cA$, and the same holds for the analogous pullback
 of $R_h$ from the second factor.
 
 Taken together, these establish our claim.
 \end{proof}
 
 \begin{cor}\label{cor29}
 The zero scheme of a general section of $E$ is smooth.
 \end{cor}
 \begin{proof}
 Nonsingularity of a section is a generic property. We proved in Theorem \ref{thm41} that
 there are sections with a smooth zero scheme in the image of $\psi$. Hence the
 same holds for the general section, and the general section is not in the image of $\psi$.
 \end{proof}

 \section{Enriques surfaces}\label{enriques}

 \subsection{Enriques surface basics}
 General reference for Enriques surfaces: \cite{D}.\par
 
 Within the Enriques-Castelnuovo classification of complex algebraic surfaces,
 Enriques surfaces are part of the surfaces of Kodaira dimension $0$ (equivalently, with trivial or torsion canonical bundle),
 alongside  abelian, K3 and hyperelliptic (sometimes also called bielliptic) surfaces.
 They are characterized by $q=p_g=0$ and $K^{\otimes 2}=\cO$.
 
 The  fundamental group of a complex  Enriques surface 
 is isomorphic to $\bZ/2\bZ$ and the universal double cover is a K3 surface.
 (see \cite[V\!I\!I\!I 15]{BPV}).\par
 
 When Bombieri and Mumford extended the classification of algebraic surfaces to characteristic $p>0$, they
 discovered new phenomena in characteristic $2$. They redefined Enriques surfaces as surfaces $X$ with
 numerically trivial canonical bundle and second Betti number $B_2=10$. They found that Enriques surfaces
 in characteristic $2$ may be divided in three classes as follows:
 \begin{enumerate}
 \item[(i)]
 \emph{classical} Enriques surfaces satisfy $h^1(\cO_X)=0$, hence $K_X\not\cong\cO_X$, but $K_X^{\otimes 2}\cong\cO_X$,
 \item[(ii)]
 \emph{ordinary} Enriques surfaces (sometimes called singular or $\mu_2$-surfaces) with $h^1(\cO_X)=1$, $K_X\cong\cO_X$, and Frobenius acts bijectively on $H^1(\cO_X)$,
 \item[(iii)]
\emph{ supersingular} Enriques surfaces (or $\alpha_2$-surfaces) with $h^1(\cO_X)=1$, $K_X\cong\cO_X$, and Frobenius on $H^1(\cO_X)$ is zero.
 \end{enumerate}
 (Their original Bombieri-Mumfod  terminology for 'ordinary' was 'singular';  Dolgachev uses 'ordinary'
 which nowadays is or is likely to become  standard)
 \par
 
 Enriques surfaces in positive characteristic continue to have a special double cover. For
 classical Enriques surfaces in characteristic $\ne 2$ and for ordinary Enriques surfaces in characteristic $2$
 the double cover is an unramified map from a K3 surface, while for classical and supersingular surfaces in
 characteristic $2$ the map is purely inseparable, and the covering surface (which is either K3 or rational)
 has singularities.

\medskip
To measure positivity of line bundles, one uses
\begin{equation*}
\Phi(D)=\min \{ |D\cdot f| : f^2=0 \text{ in } \Num(S) \}.
\end{equation*}

A big and nef primitive divisor $D$ such that $D^2=4$, $\Phi(D)=2$ is called a \emph{Cossec-Verra polarization}. 
Every Enriques
surface possesses a Cossec-Verra polarization \cite[3.4.2]{D}. However, this polarization may not be ample; in fact,
there exist ``extra-special'' Enriques surfaces in characteristic $2$ with no ample Cossec-Verra polarization
\cite[3.4]{Lie}.
Their non-ample locus is a union of smooth rational curves with self-intersection $-2$. 
There is a birational contraction $\pi\colon S\to X$ such that $D$ induces an ample divisor 
on the Gorenstein surface $X$. The singularities of $X$ consist of a finite number of rational double points.

\subsection{Enriques zero-sets}
Enriques surfaces occur as zero-sets of sections of our bundle $E$:
\begin{thm}\label{thm43}
The zero scheme of a general section $s\in H^0(E)$ is a nonsingular
surface $S$ with the following properties:
\begin{enumerate}
\item[\textup{1.}]
$\chi\cO_S=1$ and $K_S=\cO_S$, i.e., $S$ is a non-classical
Enriques surface.
\item[\textup{2.}]
The cycle class of $S$ in $\bP^2\times\bP^2$ is $4L^2+5Lh+4h^2$; i.e.,
the line bundles $L_S=\cO_S(L)$ and $h_S=\cO(h)$ have the following
intersections: $L_S^2=h_S^2=4$ and $L_S\cdot h_S=5$.
\item[\textup{3.}]
$L_S$ and $h_S$ are Cossec-Verra polarizations, but may not be ample
\emph{(}for an example, see \emph{\ref{ex48}} below\emph{)}.
\item[\textup{4.}]
The embedding of $S$ in $\bP^8$ under the Segre embedding of $\bP^2\times\bP^2$
is a linear projection of the image of $S$ in $\bP^9$ under the complete linear system
$\vert L_S+h_S\vert$.
\item[\textup{5.}]
The homogeneous ideal of $S$ is generated by $3$ polynomials each of bidegrees $(2,3)$
and $(3,2)$, and one additional generator of bidegree $(3,3)$.
\end{enumerate}
\end{thm}
\begin{proof}
Nonsingularity of the zero scheme of the general section was proved in Corollary \ref{cor29}.
$K_S$ is the restriction of $c_1E\otimes\omega_{\bP^2\times\bP^2}^{-1}=\cO$,
and $\chi\cO_S=\chi\cO-\big(\chi E(-3L-3h)-\chi\cO(-3L-3h)\big)=1$. According to the
extension of the Enriques-Kodaira classification of surfaces by Bombieri and Mumford 
\cite{bombieri-mumford-III},
$S$ is an Enriques surface with trivial canonical bundle.

The cycle class of $S$ agrees with $c_2E$, and the intersection of classes on $S$ agree
with the intersections on $\bP^2\times\bP^2$.

$L_S$ and $h_S$ are basepointfree linear systems of degree $4$, hence by
\cite[2.4.11,2.4.14]{D}, we must have $\Phi(L_S)=\Phi(h_S)= 2$, i.e., $L_S$ and
$h_S$ are Cossec-Verra polarizations.

As $h^0(\cO_S(L+h))=10$ and $h^0(\cO_{\bP^2\times\bP^2}(L+h))=9$, the embedding of $S$ inside the
Segre image of $\bP^2\times\bP^2$ is a linear projection.
\end{proof}

 \subsection{linear systems on Enriques surfaces}
We wish to identify the lines bundles $L_S, h_S$ in terms of the structure of 
 the Neron-Severi group of an Enriques surface. $\NS(S)$  has rank $10$
 and its torsion is generated by $K_S$, hence is $\Z_2$ in the classical
 case and trivial otherwise. The intersection
 form on $\Num(S)=\NS(S)/(K_S)\simeq\Z^{10}$ is unimodular, even and has signature $(1,9)$, hence
 is isomorphic to $H\oplus(-E_8)$. 

In \cite[1.5.3]{D} Dolgachev constructs a specific basis $\omega_0,\dotsc,\omega_9$
of $\Num(S)$ and determines the corresponding intersection matrix.
 
We can identify candidates for the classes of $L_S$ and $h_S$ inside $\Num(S)$ as follows:

The basis elements $\omega_1,\omega_2$ generate a primitive subgroup
with intersection matrix $\big(\begin{smallmatrix} 4 & 9 \\ 9 & 18\end{smallmatrix}\big)$ \cite{D}.
The vectors $\omega_1$, $\omega_2-\omega_1$ define an alternative basis of this subgroup
and have the desired intersection form
$\big(\begin{smallmatrix} 4 & 5 \\ 5 & 4\end{smallmatrix}\big)$.

%
%

If the classes $\omega_1$ and $\omega_2-\omega_1$ are both nef (e.g., if $S$ does not
contain any $(-2)$-curves), then they are 
Cossec-Verra polarizations, and they combine to yield a basepointfree morphism 
$S\to \bP^2\times\bP^2$.
\section{Detecting supersingularity}\label{ss-sec}

Our next task will be to identify ordinary resp. supersingular Enriques surfaces
among the zero schemes of the sections of $E$. Our plan will be to identify
suitable elliptic curves on $S$ and use the following criterion:

\emph{Suppose $C\subset S$ is a curve such that the restriction $\cO_S\to\cO_C$
induces an isomorphism $H^1(\cO_S)\to H^1(\cO_C)$. Then $S$ is supersingular
if and only if $C$ is supersingular.}


\begin{prop}\label{prop43}
Let $s\in H^0(E)$ be a section vanishing on a nonsingular Enriques surface $S$.
Then $S\cap\cA$ is a complete intersection $D_L\cap D_h\cap\cA$ for
uniquely  determined divisors $D_L\in \vert 3L\vert$, $D_h\in \vert 3h\vert$.
\end{prop}
\begin{proof}
Since 
$E\otimes\cO_A\cong\cO_\cA(3L)\oplus\cO_\cA(3h)$ of $E$ on $\cA$ 
(\cite{p2p2-I}, Corollary 23), there 
exist uniquely determined divisors $D_{L\cA}\in\vert 3L_\cA\vert$, $D_{h,\cA}\in\vert 3h_\cA\vert$ such that
$S\cap\cA=(D_{L,\cA}\cap D_{h,\cA})\cap \cA$. Furthermore, $\cA$ is a  divisor  in
$|L+h|$ on  $\bP^2\times\bP^2$,
and therefore the divisors in the linear systems $\vert 3L_\cA\vert$ and $\vert 3h_\cA\vert$ on $\cA$ are
restrictions of uniquely determined 
effective divisors $D_L$, $D_h$ from the corresponding linear systems on $\bP^2\times\bP^2$.
\end{proof}


In the following we write $D_L$, $D_h$ for both the divisor in $\P^2$, its inverse
image  on $\bP^2\times\bP^2$, and the
cubic polynomial defining either one. This should lead no confusion.

\begin{prop}\label{prop46}
Let $C_L\subset S$ be the residual of $S\cap\cA$ in $S\cap D_L$.
\begin{enumerate}
\item[\textup{1.}]
$C_L$ is an effective divisor in $\vert 2L-h\vert_S$.
\item[\textup{2.}]
We have $(2L-h)^2=0$, $h^0(\cO_S(2L-h))=1$ and $h^1(\cO_S(2L-h))=h^2(\cO_S(2L-h))=0$.
Hence $C_L$ is a half-fiber (in the sense of \cite[2.2.9]{D})
\item[\textup{3.}]
$C_L$ is connected, $h^1(\cO_{C_L})=1$, and
the map $H^1(\cO_S)\to H^1(\cO_{C_L})$ is bijective.
\item[\textup{4.}]
We have $p_{L*} C_L\in|3L|\subset\bP^2_L$ 
\end{enumerate}
\end{prop}
\begin{proof}
We calculate $C_L=3L-(L+h)=2L-h$ for the divisor class of $C_L$.
The second statement follows from the intersection pairing on $S$ and
the cohomology of $E$.
The third follows from the first (and Serre duality) upon considering the cohomology of
$0\to \cO_S(-2L+h)\to\cO_S\to\cO_{C_L}\to 0$.

The cycle class of $C_L$ in $\bP^2\times\bP^2$ is
$(2L-h)\cdot(4L^2+5Lh+4h^2)=6L^2h+3Lh^2$, hence $C_L.L=3$.
\par

Surjectivity of $H^0(\cO_S(L))\to H^0(\cO_C(L))$ follows from the vanishing of $H^1(\cO_S(-L+h))$.
\end{proof}

\begin{prop}\label{prop47}
Notation as in Proposition \ref{prop46}.
Let 
\begin{equation*}
\begin{tikzcd}[column sep=small]
C_L  \ar[r,"f'"] & D'_L \ar[r,"g"] & D_L\subset\P^2
\end{tikzcd}
\end{equation*}
be the Stein factorization of the projection $C_L\to D_L$.
Thus $f'$ has connected fibres and $g$ is finite. Then the following hold:
\begin{enumerate}
\item[\textup{1.}]
$f'$ induces an isomorphism between $H^1(\cO_{C_L)}$ and $H^1(\cO_{D'_L)}$.
\item[\textup{2.}]
$D'_L$ and $D_L$ have the same number of irreducible components with
multiplicities, and $g$ is generically bijective for each irreducible component of 
$D'_L$ and its image.
\end{enumerate}
\end{prop}

\begin{proof}
We need to investigate the structure of the divisor $C_L$.

\begin{enumerate}
\item[(1)] The components:
Every effective divisor class $C$ with $C^2\ge 0$ can be expressed as 
\begin{equation*}
C \sim C'+\sum m_iR_i, \qquad (m_i\ge 0)
\end{equation*}
where $C'$ is nef and the $R_i$ are $(-2)$-curves \cite[2.3.3]{D}.
Since $C_L$ is unique in its divisor class by Proposition \ref{prop46}, the above
decomposition is a decomposition of divisors, and not only
of divisor classes.

The potential structures for numerically effective 
half-fibers $C'$ (note that $C'$ is automatically an indecomposable divisor of canonical type
in the sense of Mumford)
have been classified \cite[2.2.5]{D}.
If $C'$ is irreducible, then
(being of arithmetic genus $1$) it is an elliptic curve, or a nodal or cuspidal rational curve.
If $C'$ is reducible, then its components are nonsingular rational curves 
of self-intersection $-2$.

\item[(2)] Vertical and horizontal components:
We can divide the components of $C_L$ into horizontal and vertical (fibral) components. The 
$f$- images of the former are $1$-dimensional, while the latter are each mapped to a point.

Vertical components are $(-2)$-curves on $S$ contracted by the Cossec-Verra polarization $L_S$.
There is a Gorenstein surface $X$ with a finite number of rational double points, and a contraction 
$\pi\colon S\to X$ such that $L_S$ induces an ample line bundle on $X$; in particular 
we note $R^1\pi_*\cO_S=0$
\cite[2.4.16]{D}.

The cycle class of $C_L$ in $\bP^2\times\bP^2$ is
$(2L-h)\cdot(4L^2+5Lh+4h^2)=6L^2h+3Lh^2$, hence the 
intersection of $C_L$ with a general divisor of $h$
consists of $3$ points (counted with multiplicity), and $C_L$ has at most 
$3$ horizontal components.
\item[(3)] Stein factorization of the projection:
Now consider the Stein factorization
\begin{equation*}
\begin{tikzcd}[column sep=small]
C_L  \ar[r,"f'"] & D'_L \ar[r,"g"] & D_L
\end{tikzcd}
\end{equation*}
where $f'$ has connected fibers and $g$ is finite.
The components of $D'_L$ are the images of the horizontal components of $C_L$.
We have $f'_*\cO_{C_L}=\cO_{D'_L}$ by construction, and we claim that $R^1f'_*\cO_{C_L}=0$.

To this end, consider the restriction to $C_L$ of the contraction $S\to X$. 
Since $\pi$ contracts all the vertical fibers that $f'$ contracts, we have
$\pi\vert_{C_L}=\pi'\circ f'$ for some morphism $\pi'$. If $R^1f'_*\cO_{C_L}$ were not $0$,
then it would be supported on a finite set of points, and $\pi'_*R^1f'_*\cO_{C_L}$ would
also be non-zero.
But the latter agrees (by Leray) with $R^1\pi_*\cO_{C_L}$, which is a quotient of $R^1\pi_*\cO_S=0$.
Therefore we must have $R^1f'_*\cO_{C_L}=0$.

We can thus conclude that $H^1( \cO_{C_L})=H^1( f'_*\cO_{C_L})$, i.e., $f'$ induces an isomorphism between
$H^1(\cO_{D'_L})$ and $H^1(\cO_{C_L})$.
\item[(4)] Let $E\subset D_L$ be an irreducible component of degree $d\le 3$. The $1$-cycle
$E\cap D_h\cap \cA$ has the class $dL\cdot 3h\cdot (L+h)=3d(L^2h+Lh^2)$ and is contained in 
$S\cap\cA$.
The residual of this cycle in $S\cap E$ has the class
$(5dL^2h+4dLh^2)-3d(L^2h+Lh^2)=2dL^2h+dLh^2$, hence intersects a general divisor in $\vert L\vert$
in $d$ points, counted with multiplicity. This implies that $C_L$, and hence $D'_L$, contains
exactly one horizontal irreducible component mapping to $E\subset D_L$.

%
\end{enumerate}

We have now reached to the following situation: There is a scheme $D'_L$ and a
finite map $D'_L\to D_L$ that is generically bijective for each irreducible component, with the
same multiplicities.
Each component of $D'_L$ is the isomorphic image of a subscheme of $X$.

In particular, the map $g$ from (3) is locally an isomorphism over each nonsingular point of $D_L$.
\end{proof}

There is a finite list of isomorphism classes for $D_L$ (see Example \ref{ex17}) and we will inspect
them individually.

\begin{prop}\label{prop48}
Notation as in  Propositions \ref{prop46}, \ref{prop47}. Then
\begin{enumerate}
\item[\textup{1.}]
$D_L$ is reduced.
\item[\textup{2.}]
If $D_L$ is irreducible, then $g\colon D'_L\to D_L$ is an isomorphism.
\item[\textup{3.}]
If $D_L$ has several irreducible components, then $g\colon D'_L\to D_L$ is also an isomorphism.
\end{enumerate}
\end{prop}
\begin{proof}
1. We argue by contradiction, and assume that $D_L$ has a non-reduced component. 
By Proposition \ref{prop47} this means that $D'_L$ consists of a double line together with another line intersecting it in one point,
or $D'_L$ is a multiplicity $3$ structure on a line. 

First assume that $D'_L$ consists of a double structure $E_1$ on a line $E$, and a single line $E'$.
The ideal sheaf of $E$ on $E_1$ has square $0$, hence is isomorphic to the line bundle
$\cO(a)$ on $E$ where $a=-E^2=2$ (intersection taken on $S$). This implies that $h^0(\cO_{E_1})\ge h^0(\cO_E(2))=3$.

Now consider the Mayer-Vietoris sequence
\begin{equation*}
0 \to H^0(\cO_{D'_L}) \to H^0(\cO_{E_1}) \oplus H^0( \cO_{E'}) \to H^0(\cO_{E_1\cap E'}):
\end{equation*}
As $h^0(\cO_{E_1\cap E'})\le 2$, we conclude that $h^0(\cO_{D'_L})\ge 4-2=2$.
However, by construction we have $h^0(\cO_{D'_L})=h^0(\cO_{C_L})=1$, hence
this contradiction implies that the assumed double structure cannot exist.

A similar argument rules out the case where $D_L$ is a triple structure on a line.

2. If $D_L$ is as smooth elliptic curve, then the projection $D'_L\to D_L$ is automatically an isomorphism.
If $D_L$ is a singular rational curve, then $D'_L$ must be irreducible, rational with $h^1\cO_{D'_L}=1$.
The last condition implies that $D'_L$ is singular, and it is immediate that the singularity must be of the
same type as the singularity of $D_L$. 

3. We already know that $D'_L$ and $D_L$ have the same number of components, and these are
smooth rational curves. Furthermore, any intersection of two components in $D'_L$ has to lie above
an intersection of the corresponding components in $D_L$.

There are $4$ different cases to consider: a line and a conic intersecting in two points or touching in one point,
three lines intersecting in one point or in three separate points.

In each case, if an intersection of two components of $D_L$ were not mirrored in $D'_L$, then 
we would have $h^1(\cO_{D'_L})=0$ in contradiction to the results of the previous proposition.
\end{proof}

 \begin{example}\label{ex48}
 Consider the Enriques surface defined by (notation as in \cite{p2p2-I},  Theorem 27) \\
 $w=\begin{pmatrix}b+c&b+c&x+z&a+c&z&y&1  \end{pmatrix}$.
 
 The unique divisor $C_L\in \vert 2L-h\vert$ has three components:
 two vertical lines over the points $a=b, b=\alpha_ic$ (where $\alpha_1,\alpha_2$ are
 the solutions of the quadratic equation $X^2+X+1=0$), and
 an ordinary elliptic curve whose projection is
 $D_L=a^2b+b^3+a^2c+abc+b^2c+ac^2+c^3$.
 
 The unique divisor $C_h\in \vert 2h-L\vert$ has four components:
 three vertical lines over the points $y=z,x=0$ and $x=y+z,z=\alpha_iy$ ($i=1,2$), and
 an ordinary elliptic curve whose projection is
 $D_h=y^3+x^2z+xyz+xz^2+z^3$.
 
 In this example neither Cossec-Verra polarization $\cO_S(L)$ nor $\cO_S(h)$ is 
 ample, as it contracts at least $2$ resp. $3$ rational curves.
 \end{example}

Finally we are ready for the main theorem of this section. 
 
 \begin{thm}\label{thm49}
 Let $s\in H^0E$ be a section vanishing on a nonsingular Enriques surface $S$
 with trivial canonical bundle. 
 The following conditions are equivalent:
 \begin{enumerate}
 \item[\textup{(i)}] 
 $S$ is supersingular.
 \item[\textup{(ii)}]
 $s$ is in the image of the canonical map
 \begin{equation*}
 \psi\colon \Big(H^0(E(-L))\otimes H^0(\cO(L))\Big)
 \oplus
 \Big(H^0(E(-h))\otimes H^0(\cO(h))\Big)
 \to H^0(E).
 \end{equation*} 
 \item[\textup{(iii)}] 
 The restriction  $s|_{\cA}\in H^0(E|_\cA)=H^0(\O_{\P^2_L}(3)\oplus\O_{\P^2_h}(3))$ 
 corresponds to a pair of cubic polynomials $D_L,D_h$
 defining possibly degenerate supersingular cubic curves \emph{(}as explained in \emph{\ref{ex17})}.
 \item[\textup{(iv)}] 
 The homogeneous bigraded ideal of $S$ is not generated by its elements
 in bidegrees $(3,2)$ and $(2,3)$; i.e. there is an additional generator in
 bidegree $(3,3)$.
 \end{enumerate}
 \end{thm}
 \begin{proof}
 The equivalence of (ii) and (iii) is immediate from the discussion in Corollary \ref{multi-cor}.
 
 (ii) and (iv) are equivalent:
 Consider the diagram
 \begin{equation*}
 \begin{tikzcd}
 && V_L\oplus V_h \ar[d] \\
 0 \ar[r] & \cO \ar[r] & E \ar[r] & \cI_S(3,3) \ar[r] & 0
 \end{tikzcd}
 \end{equation*}
 where $V_L=H^0(E(-L))\otimes L$ and $V_h=H^0(E(-h))\otimes h$.
 We know that $h^0(V_L\oplus V_h)=18$ and $h^0(E)=19$.

 The section $s\in H^0E$ is in the image of $H^0(V_L\oplus V_h)$ if and only if
 the morphism $H^0(V_L\oplus V_h)\to H^0(\cI_S(3,3))$ is not surjective.
 
 (i) and (iii) are equivalent:
 Recall that we identified in Proposition \ref{prop46} two (possibly reducible)
 half-fibers on $S$ that are supersingular if and only if $S$ is.
 Furthermore these divisors are contained in the intersection of $S$ with the 
 hypersurfaces of $\bP^2\times\bP^2$
 cut out by the pullbacks of the polynomials $D_L$ resp. $D_h$.
 
 The equivalence follows if we can show that the projection $C_L\to D_L$
 (or the projection $C_h\to D_h$) induces an isomorphism on $H^1(\cO)$.
 
But this is exactly what we established in Proposition \ref{prop48}.
 \end{proof}
 
 \begin{rem}
 Theorem \ref{thm49} probably extends to Enriques surfaces with rational double points.
 In this case, $L$ and $h$ restrict to Cossec-Verra polarizations but $\vert L+h\vert$
 is not ample. 
 \end{rem}
 
 
There is a quasi-projective dense open subscheme $U\subset \bP H^0(E)=\bP^{18}$ corresponding to
sections of $E$ whose zero scheme is a nonsingular Enriques surface. 
Theorem \ref{thm49} showed that the closed points corresponding to supersingular Enriques surfaces
form a linear subvariety of codimension $1$.

We now want to show a slightly stronger result:

Consider the universal family $\pi\colon X_U\to U$, which is a subscheme 
of $\bP^2\times\bP^2\times\bP^{18}$.
By Grauert's theorem, $R^1\pi_* \cO_{X_U}$ is a line bundle on $U$, and 
the relative Frobenius morphism on $X_U$, viewed as a map 
\begin{equation*}
(R^1\pi_*)(F^*)\colon R^1\pi_* \cO_{X_U}\to R^1\pi_* \cO_{X_U},
\end{equation*}
i.e. a function on $U$, defines a principal divisor of $U$, 
a priori potentially non-reduced,  which corresponds
to the subfamily of supersingular Enriques surfaces in our family. 

\begin{cor}
The divisor of supersingular Enriques surfaces is reduced and nonsingular.
\end{cor}
\begin{proof}
Corollary 29 in Appendix 1 to \cite{p2p2-I} provides an explicit parametrization of the family of all sections,
and we can use it to compute the action of Frobenius on the direct image sheaf.

\textbf{Step 1}:
Every section $s\in H^0(E)$ defines by restriction to $\cA$
two sections $p\in H^0(\cO_{\bP^2}(3))$, $q\in H^0(\cO_{\bP^2}(3))$, and we claim that the subscheme
of supersingular Enriques surfaces agrees with the subscheme of supersingular cubic curves defined by $p$.

To see this, we go back to our previous analysis.

In Proposition \ref{prop46} we constructed an effective divisor $C\subset X$ 
that represents an elliptic half-fiber on each surface.
The construction generalizes to the family over $U$ and induces a divisor $C_U\subset X_U$ such that
the morphism
\begin{equation*}
R^1\pi_*\cO_{X_U}\to R^1\pi_*\cO_{C_U}
\end{equation*}
is bijective (as the map $H^1(\cO_{X_u})\to H^1(\cO_{C_u})$ is bijective for each $u\in U$)
and commutes with the relative Frobenius.

Furthermore, by Proposition \ref{prop43}) there exists an effective divisor $D_U\in H^0\big( pr_1^*\cO_{\bP^2}(3)\big)$ 
such that the morphism
\begin{equation*}
R^1\pi_*\cO_{C_U}\to R^1\pi_*\cO_{D_U}
\end{equation*}
is bijective and commutes with relative Frobenius (we proved this for the fibers in Propositions \ref{prop47} and \ref{prop48}).

Now the divisor $D_U$ corresponds to the section $p$ mentioned in the beginning of step 1.

\textbf{Step 2}:
In order to complete the proof, we need to investigate the relative Frobenius morphism
\begin{equation*}
(R^1\pi_*)(F^*)\colon R^1\pi_* \cO_{D_U}\to R^1\pi_* \cO_{D_U}
\end{equation*}
for the family of cubic curves $D_U\in H^0\big( pr_1^*\cO_{\bP^2}(3)\big)$ over $U$, defined by $p$.

According to Corollary 29 of \cite{p2p2-I}
we can write $p = f_1\cdot a^2 + f_2\cdot b^2 + f_3\cdot c^2 + h\cdot abc$
where $\bP^2=\Proj (k[a,b,c])$, and
$f_i=\lambda_{i,a} a+ \lambda_{i,b} b+ \lambda_{i,c} c$ (for $i=1,2,3$), $h=\lambda_{abc}$
for suitably indexed coordinate functions $(\lambda_*)$ on $\bP H^0(E)$.

We now apply the discussion from the proof of \cite[4.21]{hart}: 
Identifying $H^1(\cO_{D_U})$ with \\ $H^0(\cO_{\bP^2_U}(3))$, represented by the one-dimensional 
vector space with natural basis $(abc)^{-1}$ of $H^0(\cO_{\bP^2_U}(3))$, we obtain that the image of $H^1(\cO_{D_U})$ 
under the relative Frobenius is
$(abc)^{-1}$ times the coefficient of $abc$ in $p$, hence agrees with
with the coordinate function $\lambda_{abc}$.

Therefore we can conclude that the subscheme of $U$ corresponding to supersingular Enriques surfaces is defined
by the vanishing of a coordinate function on $\bP H^0(E)$, hence is reduced and nonsingular.
\end{proof}


 \section{A reducible model}\label{reducible-sec}
 Here we discuss a reducible normal-crossing limit of the general zero-set $Y$.
 \begin{prop}
 A special zero-set of $E$ has the form
 \[(s)_0=Y_1\cup Y_2\]
 where $Y_1\simeq\P^2$, $Y_2$ is an elliptic ruled surface and $Y_1\cap Y_2$ is
 a smooth cubic in $Y_1$ and an unramified bisection of $Y_2$.
 \end{prop}
 \begin{proof} Consider the modification sequence
 \[\exseq{(p_2^*F^*Q_h)(L)}{E}{\O_\cA(3h)}.\]
 Note that $\cA$ induces a duality
 \eqspl{duality-eq}{\P^2_L\simeq\P^{2*}_h, \P^2_h\simeq\P^{2*}_L.}
 For a section $s$ of $E$ coming from a general section $s_1\in H^0((p_2^*F^*Q_h)(L))$, the zero set takes the form
 \[(s)_0=Y_1\cup Y_2.\]
 where 
 \[Y_1=(s_1)_0, Y_2\in|\O_\cA(3L)|,\]
 \[[Y_1]=c_2(F^*Q_h(L))=4h^2+2hL+L^2, [Y_2]=(L+h).3L=3L^2+3Lh.\] 
 Now $s_1$ has the form $u_0v^2_0+u_1v^2_1+u_2v^2_2$ where $u.$ are homogeneous coordinates
 on $\P^{2}_L$ , and the $v.$ are a basis for $H^0(Q_h)$ vanishing respectively
 at $p_0, p_1, p_2$ that are the vertices of the coordinates triangle. 
 Clearly $Y_1$ is smooth off the $p_i$. At $p_0$, we may assume that, with respect to a suitable local basis for $Q_h$,
 $v_0^2$ has the form $(x^2, y^2)$ while $v_1^2=(1,0), v_2^2=(0,1)$,
  and the local equations for $Y_1$ are 
 $x^2u_0+u_1=0, y^2u_0^2+u_2=0$. Hence $Y_1$ is smooth at $p_0$
 and likewise at the other $p_i$. Thus $Y_1$ is smooth everywhere. \par
 Note that  because $\cA$ is smooth over $\P^2_L$, this argument
  also proves that $W=Y_1\cap \cA$ is smooth.
 \par
 Note that $Y_1$ is anticanonical, i.e. $K_{Y_1}=\O_{Y_1}(-L-h)$
 and because  has class $[Y_1]=4h^2+2hL+L^2$,  $Y_1$ maps birationally to $\P^2_h$
 and under the  Segre embedding $Y_1\subset \P^2_L\times\P^2_h\to \P^8$, which
 is also an anticanonical map, $Y_1$ has degree $9=(-K_{Y_1})^2=h^0(K_{Y_1})-1$.
 It follows that $Y_1$ must project isomorphically to $\P^2_h$ and its image in $\P^8$ is a projection of its full anticanonical
 image in $\P^9$.
 \par 
 As for $Y_2$, it  is a ruled surface of the form $\P(Q_L|_Z)$ where $Z\subset\P^2_L$ is the cubic
 with equation, in the above notation, $u_0^tx^2_0+u_1^tx^2_1+u_2^tx^2_2$
 where $x_i$ are the coordinates dual to $v_i$ by the duality \eqref{duality-eq}, hence 
 $Z$ is smooth. 
 $Y_2$ may be obtained from $Z\times\P^1$ by elementary transformations (blowing up
 a point and blowing down the proper transform of its fibre) at 3 collinear points.\par
 As for the intersection $W=Y_1\cap \cA$, it is a zero-set of $F^*Q_h(L)|_\cA$ which fits in an exact sequence
 \[\exseq{\O_\cA(2h-L)}{F^*Q_h(L)|_\cA}{\O_\cA(3L)}\]
 hence $W$ projects to $Z$ in $\P^2_L$. Moreover
 \[ [W]=[Y_1].(L+h)=6h^2L+3L^2h\]
 hence $W.L=6$, hence $W$ maps with degree 2 to $Z$. On the other hand $W.h=3$
 sp $W$ maps isomorphically to a cubic in $\P^2_h$. This also implies that the map $W\to Z$
 is unramified.\par
 Note that the fact that $\omega_{Y_1\cup_W Y_2}=\omega_{Y_1}(W)\cup\omega_{Y_2}(W)$
 is the trivial bundle a priori forces $W$ to be a cubic in $Y_1$ and a bisection in $Y_2$.
 \end{proof}
 The Mayer-Vietoris sequence
 \[\exseq{\O_{Y_0}}{\O_{Y_1}\oplus\O_{Y_2}}{\O_W}\]
 yields an injection $H^1(\O_{Y_0})\to H^1(\O_{Y_2})$ (induced by restriction). 
 Since $Y_2\to Z$ is a $\P^1$-bundle, we have an isomorphism of
 1-dimensional vector spaces  $H^1(\O_Z)\simeq H^1(\O_{Y_2})$
 induced by pullback. Since $H^1(\O_{Y_0})\neq 0$, we conclude
 $H^1(\O_{Y_0})\simeq H^1(\O_{Y_2})\simeq H^1(\O_Z)$.
 Both these isomorphisms are compatible with the action of Frobenius.
 As the equation of $Z$ has no $x_0x_1x_2$ term, $Z$ is supersingular by \cite{hart}, IV.4.21.
 Hence we conclude
 \begin{cor}\label{ss-cor}
 Notations as above, $Y_0$ is supersingular.
 \end{cor}
 Note that in the family of all pairs $(E, s)$ so that the zero-set $(s)_0=Y_1\cup Y_2$ as above has codimension
 10 in the family of all pairs $(E,s)$ where $(s)_0$ is a smooth or normal-
 crossing surface. On the other hand,  the locus where $(s)_0$ is supersingular is a 
 divisor. Corollary \ref{ss-cor}
 implies that this divisor is nontrivial. A priori,  a general surface
 in this divisor may have finitely many double points. However the results of
\S \ref{ss-sec}, specifically Theorem \ref{thm49},
proven  by other means, shows that a general supersingular zero-set is in fact smooth. 


\section{Moduli}\label{moduli-sec}
As we saw earlier, there is a bijection between the set of bundles $E$
and the set of smooth divisors in $|L+h|$. This bijection is in fact
an isomorphism of moduli spaces:
\begin{thm}
(i) Let $U\subset|L+h|$ be the subset parametrizing smooth divisors and let 
$\cA_U\subset U\times\P^2\times\P^2$ be the universal divisor. Then the kernel
$\cE$ of the natural surjection
\[F^*Q_h\to\O_{\cA_U}(2L)\]
is a vector bundle whose restriction $E_u$ on  $u\times\P^2\times\P^2$ for each  
$u\in U$ is indecomposable
and coincides with one of the bundles $E$ studied above
and every such bundle occurs in this way; moreover for each $u$, $\cE$ is the universal 
deformation of $E_u$.\par
(ii) There is an isomorphism between the space $V$ of pairs $(E, s)$ where $s\in H^0(E)$
has smooth zero-set and the space $V'$ of pairs $(S, \Phi_L\times\Phi_h)$ where $S$
is a smooth nonclassical Enriques surface and $\Phi_L\times\Phi_h:S\to\P^2\times\P^2$
is an embedding whose image $S'$ has $S'.L^2=S'.h^2=4, S'.L.h=5$.
\end{thm}
\begin{proof}
By construction each bundle $E$ as above corresponds to a smooth divisor $\cA$,
hence to a point $u\in U$ and conversely.
 The divisor $\cA$ depends canonically on $E$ as the degeneracy divisor of the map
\[p_h^*p_{h*}E\to E.\]
Conversely given a smooth divisor $\cA\in |L+h|$, it determines a bundle $E$ as a suitable twist
of the kernel of the canonical map
\[F^*Q_h\to\O_\cA(2L).\]
Moreover the correspondence $E\leftrightarrow \cA$ clearly extends to deformations over any local scheme $S$: given a bundle $E_S$ on $S\times\P^2\times\P^2$, we get a divisor $\cA_S$
on $S\times \P^2\times\P^2$ as the degeneracy locus of
$p_2^*p_{h*}E_S\to E_S$ and conversely given a divisor $\cA_S$ we get $E_S$ as the kernel
of $F^*Q_h\to\O_{\cA_S}(2L)$.
Therefore the deformation space of $E$ can be identified with $U$
and in particular it is smooth and 8-dimensional.\par

Now because $h^0(E)=19, h^1(E)=0$ for each $E=E_u$,
$\bP(\pi_{U*}(\cE))$ is smooth and relatively 18-dimensional over $U$.
Let $V\subset \bP(\pi_{U*}(\cE))$ be the open subset of sections with smooth zero-sets.
The $V$ is   the deformation space of  pairs $(E, k s)$  ($k$ the ground field),
where $s\in H^0(E)$ has smooth (Enriques) zero-set, and $V$ is smooth and 26 dimensional.\par
On the other hand the deformation space $V'$ of pairs $(Y, \Phi)$ where $Y$ is an Enriques surface
(10 moduli) and $\Phi:Y\to\P^2_L\times\P^2_h$ is an embedding
with $L^2=h^2=4, L.h=5$ (uniquely determined
by $(Y, L, h)$ up to an automorphism of $\P^2_L\times\P^2_h$)
is smooth of relative dimension 16 over the 10-dimensional moduli space $M$
of Enriques surfaces. 
There are mutually inverse maps
\[V\stackrel{\alpha}{\to}V'\stackrel{\beta}{\to}V\]
where $\alpha$ is the zero-set map and $\beta$ is given by the Serre construction.
Hence $\alpha, \beta$ are isomorphisms.
We have established an isomorphism between deformation spaces
 of $(E, s)$ and of $(Y, \Phi)$. This implies in particular that a general zero set $Y$ is general in
 the moduli of Enriques surfaces, and that  for a general Enriques surface
 in characteristic 2 there exist divisors $L, h$ with $L^2=h^2=4, L.h=5$ such that 
 $\Phi_L\times \Phi_h:Y\to\P^2\times\P^2$ is an embedding.\end{proof}

\section{From surface to bundle}\label{embeddings}

The purpose of this section is to prove Theorem \ref{surface-to-bundle-thm} below
which asserts that every non-classical
Enriques surface of bidegree (4,4) in $\P^2\times\P^2$ occurs as
a zero-set one of the 'monadic' bundles
constructed above.

Note that the Chow group of $2$-cycles on $\P^2\times\P^2$ has a basis
that consists of $L^2$, $Lh$ and $h^2$. \par
Also, note at the outset that given a nef and big line bundle $L$ on a 
nonclassical  Enriques surface $S$ (i.e. $K_S=0$), we have
$h^1(L)=0$ \cite[2.1.7]{D}, hence $h^0(L)=\chi(L)=L^2/2+1$.
We assume \footnote{Recall the analogy with $\P^4$: All
smooth surfaces in $\P^4$ are linearly normal except the Veronese surface.} 
in the following that $S\subset\P^2\times\P^2$ 
is 'bilinearly normal', i.e.  it is not a 'linear projection' from a higher-dimensional $\P^r\times\P^s$.
This is equivalent to $\cO(L)$, $\cO(h)$ restricting to Cossec-Verra polarizations on 
$S$ (i.e., $L_S^2=h_S^2=4$, and $\Phi(L_S)=\Phi(h_S)=2$).


Setting $\lambda=L_S\cdot h_S$, the self-intersection formula for $S$ (see
\cite[App. A 4.1.3]{hart}) now shows that $\lambda^2-9\lambda+20=0$, 
hence $\lambda=4$ or $\lambda=5$.

We first analyze the case $\lambda=4$. Then we have $(L_S-h_S)^2=0$, and the Hodge index theorem implies that
$L_S-h_S$ is numerically trivial, and since this class must be nontrivial, necessarily $S$ is a classical
Enriques surface and $h_S=L_S\otimes K_S$.

Conversely, given a Cossec-Verra polarization $L$ on a classical Enriques surface $S$,
there is a birational contraction $f\colon S\to X$ to a Gorenstein surface $X$ such that
$\vert L\vert\times\vert L\otimes K_X\vert$ defines an embedding of $X$ into
$\bP^2\times\bP^2$ \cite[3.4.7]{D}.

Continuing with the case $\lambda=4$, since $K_S=h\otimes L^{-1}$, the canonical bundle is restriction of the global 
line bundle $\cO(-1,1)$ on $\bP^2\times\bP^2$. According to Serre's theorem \cite{okonek},
there exists a rank $2$ vector bundle on $\bP^2\times\bP^2$
with a section $s$ whose zero scheme is the given surface.
These bundles have been constructed by Casnati and Ekedahl \cite[6.5]{CE}
who showed that they are (up to twist) pullbacks of a bundle  from one of the factors 
$\P^2$, so they are not 'new'. The corresponding bundles on $\P^2$  are in turn
 pullbacks  of
the tangent bundle   by  quadratic maps;  
they have a minimal graded resolution $0\to \cO(-2)^{\oplus 3} \to \cO\to \cF\to 0$,
and their moduli were described by Barth \cite{barth-vbP2}. Both \cite{CE} and \cite{barth-vbP2}
are written in characteristic 0 but it seems likely that they largely extend, with
some exceptions,  to characteristic $p$. 

Now assume $\lambda=5$, i.e., we have an embedding $S\to\bP^2\times\bP^2$ of an
Enriques surface such that the class of the surface $S$ is $4h^2+5hL+4L^2$.

We do not know if there exist classical Enriques surfaces with such an embedding%
\footnote{As discussed in section 2.3, every Enriques surface with no $(-2)$-curves
admits a morphism $S\to\bP^2\times\bP^2$, and the cycle of the image will realize
this class. However, the image of $S$ under the Segre embedding $\bP^2\times\bP^2\to \bP^8$
is a linear projection from the embedding $S\to\bP^9$ defined by $L_S+h_S$. 
Hence existence depends on analyzing the projection map.}.

If $S$ is non-classical, then the canonical bundle is trivial, and there exists a rank $2$ 
vector bundle $E$ on $\bP^2\times\bP^2$ with a section $s$ whose vanishing scheme is $S$.

In the remainder of this section we want to show that the bundle $E$ belongs to the 
family of bundles that we constructed in \cite{p2p2-I}, Part 2,  as the cohomology of a monad.

\begin{lem}
$S$ uniquely determines the bundle $E$.
\end{lem}
\begin{proof}
The Serre construction  \cite{okonek}, says that given $S$, the choice of bundle $E$
corresponds to that of a line bundle  $\eta\in \Pic(\bP^2\times\bP^2)=\Z\oplus\Z$ 
that restricts on $S$ to $K_S$.
Since $L_S\cdot h_S=5, L_S^2=h_S^2=4$, the restrictions $L_S$ resp. $h_S$ of the bundles $\cO_{\bP^2}(1,0)$ and
$\cO_{\bP^2}(0,1)$ are
linearly independent in $\Pic(S)$, so the restriction map
$\Pic(\P^2\times\P^2)\to\Pic(S)$ is injective.  Thus $\eta$ is uniquely determined, hence so
is  $E$, up to isomorphism.
\end{proof}

\begin{rem}
Let $L$ be an ample Cossec-Verra polarization on a classical Enriques surface (any characteristic).
$\vert L\times( L+K_S)\vert$ determines an embedding $S\to \bP^2\times\bP^2$ \cite[3.4]{D}.

The canonical sheaf of $S$ is $2$-torsion, hence we have
\begin{equation*}
K_S \cong h_S-L_S\cong L_S-h_S\cong (2i+1)L_S-(2i+1)h_S, \forall i\in\Z,
\end{equation*}
and $S$ can be represented in infinitely many different ways as the zero section of a rank $2$ bundle.
As mentioned above, up to twist these bundles arise by pullback from one of the factors of $\P^2$.
\end{rem}

Our strategy will be to reconstruct $E$ from its (higher) direct images under $pr_2$ via the
relative Beilinson spectral sequence. To calculate the direct images, we need to make use
of the two auxiliary exact sequences $0\to\cO\to E\to \cI_S(3,3)\to 0$ and
$0\to \cI_S \to \cO\to \cO_S\to 0$, and work backwards, starting with the determination of the
higher direct images of $\cO_S$.

As mentioned earlier, the line bundles $L_S=\cO_S(1,0)$ and $h_S=\cO_S(0,1)$ are Cossec-Verra polarizations,
i.e., they are basepointfree and have self-intersection $4$. 
This implies that $h^0(L_S)=h^0(h_S)=3$ and the higher cohomology groups of $L_S$ and $h_S$ vanish.

\begin{prop}
$pr_{2,*}\cO_S\cong \cO_{\bP^2}\oplus \cO_{\bP^2}(-3)\oplus \Omega^1_{\bP^2}$.
\end{prop}
\begin{proof}
There is a birational contraction of smooth rational $(-2)$-curves $\pi\colon S\to X$ such that $\pi_*L_S$
is ample on $X$, and the induced morphism $X\to\bP^2$ is finite \cite[2.4.16]{D}.

Therefore the higher direct images of $\cO_S$ under $pr_2$ are zero, and we have 
isomorphisms $H^i (pr_{2,*}\cO_S(0,l))\cong H^i (\cO_S(0,l))$ for any $i$, $l$.

Considering the composition $S \to X\to \bP^2$, we obtain a
short exact sequence 
\begin{equation}\label{eq111}
0\to \cO_{\bP^2}\to pr_{2,*}\cO_S\to\cE\to 0, 
\end{equation}
with a locally free sheaf $\cE$ of rank $3$ on $\bP^2$.

We now use the sequence \eqref{eq111} to determine the cohomology groups $h^i\cE(l)$. Since the higher
cohomology groups of a big and nef line bundle on an Enriques surface vanish \cite[2.1.16]{D}, we find
(where empty slots mean 0)
\medskip

\begin{tabular}{c c | *{6}c }
& 2 & 3&1&&  \\
$i$ &1  &  &1 &  &  \\
& 0 &&&&3 \\
\hline
&& -1 & 0 & 1 & 2 \\
&&&& $l$
\end{tabular}
\medskip

\noindent The Beilinson spectral sequence for $\cE(1)$ shows that there is
an exact sequence $0\to \Omega^1(1)\to\cE(1)\to \cE'\to 0$ where $\cE'$ fits into
$0\to \cE'\to \cO(-1)^{\oplus 3}\to \Omega^1(1)\to 0$. The second sequence implies that
$\cE'\cong \cO(-2)$, and therefore the extension class of the first sequence is zero.
We obtain $\cE(1)\cong\Omega^1(1)\oplus\cO(-2)$.

The extension class of the sequence \eqref{eq111} vanishes as well, hence $pr_{2,*}\cO_S\cong\cE\oplus\cO$.
\end{proof}

\begin{prop}
$pr_{2,*}\cO_S(1,0)\cong \cO_{\bP^2}^{\oplus 3}\oplus \cO_{\bP^2}(-1)$.
\end{prop}
\begin{proof}
As explained in the beginning of the proof of the previous proposition, the fibers of $pr_2$
are finite with the potential exception of a finite number of smooth rational curves.

$\cO_S(1,0)$ has positive degree on these curves, and therefore we obtain (as in the previous proof)
that the higher direct images of $\cO_S(1,0)$ under $pr_2$ vanish, and there are isomorphisms
$H^i (pr_{2,*}\cO_S(1,l))\cong H^i (\cO_S(1,l))$ for any $i$, $l$.

We now determine the cohomology groups $h^i\cO_S(1,l)$ for $-2\le l\le 0$:
\medskip

\begin{tabular}{c c | *{6}c }
& 2 & 1&&&  \\
$i$ &1  &  & &  &  \\
& 0 & && 3 \\
\hline
&& -2 & -1 & 0  \\
&&& $l$
\end{tabular}
\medskip

$l=0$: Higher cohomology groups of a big and nef line bundle vanish.

$l=-1$: The line bundle $\cO_S(1,-1)$ is non-trivial and has intersection $0$ with the ample divisor $L+h$,
hence it has no sections, and its inverse also has no sections, thus (by duality) $h^2=0$. The vanishing of $h^1$ follows
from Riemann-Roch on $S$.

$l=-2$: $\cO_S(1,-2)$ has negative intersection with $L+h$, hence it has no sections. If $h^1>0$, then $\cO_S(-1,2)$
would be an elliptic pencil and thus have a half-fiber \cite[2.2.9]{D}, i.e., the divisor class would be be divisible by $2$.
But this contradicts the primitivity of this divisor class. Hence
$h^1=0$ and $h^2=1$.

Finally, the Beilinson spectral sequence provides an exact sequence
\begin{equation*}
0\to \cO_{\bP^2}^{\oplus 3} \to pr_{2,*}\cO_S(1,0) \to \cO_{\bP^2}(-1) \to 0
\end{equation*}
which necessarily splits.
\end{proof}

\begin{cor}\label{corr5}
The direct image sheaves $R^i pr_{2,*}\cI_S(l,0)$ in the range $0\le l\le 1$ are as follows:
\smallskip

\emph{
\begin{tabular}{c c | *{6}c }
& 2 & &  \\
$i$ & 1 &  $\cO_{\bP^2}(-3)\oplus\Omega^1_{\bP^2}$ & $\cO_{\bP^2}(-1)$   \\
& 0 && \\
\hline
&& 0 & 1  \\
&& $l$
\end{tabular}
}
\smallskip
\end{cor}
\begin{proof}
Using the previous two propositions, the table follows from the exact sequence 
$0\to\cI_S\to\cO\to\cO_S\to 0$ by twisting (for $l=1$)
and taking direct images under $pr_2$.

For $l=0$ we obtain
\begin{equation*}
\begin{tikzcd}[column sep=small, row sep=small]
0 \ar[r] & pr_{2,*}\cI_S \ar[r] & pr_{2,*}\cO_{\bP^2} \ar[r] \ar[d,equals] & pr_{2,*}\cO_S \ar[d,equals]  \ar[r] & R^1pr_{2,*}\cI_S \ar[r] & 0\\
&& \cO_{\bP^2} & \cO_{\bP^2}\oplus\cO_{\bP^2}(-3)\oplus\Omega^1_{\bP^2}
\end{tikzcd}
\end{equation*}
Since $H^0(\cO_{\bP^2})\to H^0(\cO_S)$ is bijective, we obtain the column for $l=0$.

For $l=1$ we obtain
\begin{equation*}
\begin{tikzcd}[column sep=small, row sep=small]
0 \ar[r] & pr_{2,*}\cI_S(1,0) \ar[r] & pr_{2,*}\cO_{\bP^2}(1,0) \ar[r] \ar[d,equals] & pr_{2,*}\cO_S(1,0) \ar[d,equals]  \ar[r] & R^1pr_{2,*}\cI_S(1,0) \ar[r] & 0\\
&& \cO_{\bP^2}^{\oplus 3} & \cO_{\bP^2}^{\oplus 3}\oplus\cO_{\bP^2}(-1)
\end{tikzcd}
\end{equation*}
Since $H^0(\cO_{\bP^2}(1,0))\to H^0(\cO_S(1,0))$ is bijective, we obtain the column for $l=1$.
\end{proof}

\begin{prop}
The direct image sheaves $R^i pr_{2,*}E(l,-3)$ in the range $-4\le l\le -2$ are as follows:
\smallskip

\emph{
\begin{tabular}{c c | *{6}c }
& 2 & &&  \\
$i$ & 1 &$\cO_{\bP^2}(-2)$ & $\Omega^1_{\bP^2}$ & $\cO_{\bP^2}(-1)$   \\
& 0 &&& \\
\hline
&& -4 & -3 & -2  \\
&&& $l$
\end{tabular}
}
\end{prop}
\begin{proof}
We propose to twist the exact sequence $0\to \cO\to E\to \cI_S(3,3)\to 0$ and
apply direct images under $pr_2$.

Twisting by $\cO(-2,-3)$ and taking direct images, we find that
\begin{equation*}
R^i pr_{2,*}E(-2,-3)\cong R^i pr_{2,*}\cI(1,0)
\end{equation*}
for any $i$, and the column for $l=-2$ follows from Corollary \ref{corr5}.

Twisting by $\cO(-3,-3)$ and taking direct images, we find
that $pr_{2,*}E(-3,-3)=pr_{2,*}\cI_S=0$, and further that there is
an exact sequence of sheaves on $\bP^2$
\begin{equation*}
\begin{tikzcd}[column sep=tiny, row sep=small]
0\ar[r] & R^1 pr_{2,*} E(-3,-3) \ar[r] & R^1 pr_{2,*} \cI_S \ar[r] \ar[d, equals] & 
R^2 pr_{2,*}\cO(-3,-3) \ar[r] \ar[d,equals] & R^2 pr_{2,*} E(-3,-3).\\
&& \cO(-3)\oplus\Omega^1 & \cO(-3)
\end{tikzcd}
\end{equation*}
Relative duality implies for the last term that
\begin{equation*}
\big(R^2 pr_{2,*} E(-3,-3)\big)^\vee 
\cong pr_{2,*} \big(E(-3,-3)\big)\otimes\cO_{\bP^2}(3)=0,
\end{equation*}
hence $R^1 pr_{2,*} E(-3,-3)$ is the kernel of an epimorphism 
$\cO(-3)\oplus\Omega^1 \to \cO(-3)$. Since there is no surjection $\Omega^1 \to \cO(-3)$,
the induced map $\cO(-3)\to\cO(-3)$ must be non-trivial, hence an isomorphism, and
we conclude that $R^1 pr_{2,*} E(-3,-3)\cong \Omega^1$.

This establishes the column for $l=-3$.

Finally, relative duality implies that
\begin{equation*}
\big(R^i pr_{2,*} E(-4,-3)\big)^\vee\cong \big(R^{2-i} pr_{2,*} E(-2,-3)\big)\otimes\cO_{\bP^2}(3),
\end{equation*}
hence the entries in the column for $l=-2$ imply the entries for $l=-4$.
\end{proof}
We can now state and prove the main result of this section:
\begin{thm}\label{surface-to-bundle-thm}
 Let $S$ be a smooth nonclassical Enriques surface in $\P^2\times\P^2$
such that $L^2.S=h^2.S=4$. Then $S$ is a zero set of a rank-2 vector bundle 
$E$ that  is the homology sheaf of a monad of the form
\begin{equation*}
0\to \cO(L+h) \to Q_L(L) \otimes Q_h(h) \to \cO(2L+2h) \to 0.
\end{equation*}
\end{thm}
\begin{proof}
Using the table from the previous proposition, the relative Beilinson spectral sequence
implies that $E(-2,-3)$ is isomorphic to the middle homology of a complex with the following terms
\begin{equation*}
0 \to \cO(-1,-2) \to pr_1^*\Omega^1(1) \otimes pr_2^*\Omega^1 \to \cO(0,-1) \to 0.
\end{equation*}
The result follows by twisting with $\cO(2,3)$ and noting that 
$Q_L=pr_1^*\Omega^1(2)$, $Q_h=pr_2^*\Omega^1(2)$, $\cO(L)=\cO(1,0)$ and $\cO(h)=\cO(0,1)$.
\end{proof}
\input{p2p2-II9.2.bbl}

\end{document}

%% file: myheader3.tex
\def\inv{^{-1}}

 \DeclareMathOperator{\Hom}{Hom}

\def\cA{\mathcal A}
\def\cO{\mathcal O}
\def\bP{\mathbb P}

\def\cF{\mathcal{F}}
\def\cE{\mathcal E}

\def\refp #1.{(\ref{#1})}
\newcommand\carets [1]{\langle #1 \rangle}

\newcommand{\Cal}[1]{\mathcal #1}

\def\sbr #1.{^{[#1]}}
\def\sfl #1.{^{\lfloor #1\rfloor}}

\def\inv{^{-1}}
\def\?{{\bf{??}}}

\def\Proj{\textrm{Proj}}

\def\Y{\Cal Y}

\def\P{\mathbb P}

\def\Z{\mathbb Z}

\def\Spec{\text{\rm Spec} }

\def\O{\mathcal O}

\def\g{\mathfrak g}

\def\1/2{\frac{1}{2}}

\def\I{\mathcal{ I}}
\def\Y{\Cal Y}

\def\im{\text{im}}

\def\2{{[2]}}
\def\l{\ell}
\def\nl{\newline}

\def\<{\langle}
\def\>{\rangle}

\def\im{\text{im}}
\def\2{{[2]}}
\def\l{\ell}

\def\Proj{\text{Proj}}
\def\scl #1.{^{\lceil#1\rceil}}
\def\spr #1.{^{(#1)}}
\def\sbc #1.{^{\{#1\}}}

\def\subpr#1.{_{(#1)}}

\def\beq{\begin{equation*}}
\def\eeq{\end{equation*}}

\def\g3{{\Gamma\spr 3.}}

\newcommand{\eqspl}[2]{
\begin{equation}\label{#1}
\begin{split}
#2\end{split}\end{equation}}

\newcommand{\exseq}[3]{
0\to #1\to #2\to #3\to 0
}

\newcommand{\beginalphaenum}{
\begin{enumerate}\renewcommand{\labelenumi}{ }
\item \begin{enumerate}
}

\def\eex{\end{rm}\end{example}}

%% file: p2p2-II9.2.bbl
\providecommand{\bysame}{\leavevmode\hbox to3em{\hrulefill}\thinspace}
\providecommand{\MR}{\relax\ifhmode\unskip\space\fi MR }
\providecommand{\MRhref}[2]{%
  \href{http://www.ams.org/mathscinet-getitem?mr=#1}{#2}
}
\providecommand{\href}[2]{#2}